\documentclass[preprint,12pt]{elsarticle}

\usepackage{graphics}
\usepackage{graphicx}

\usepackage{amssymb}
\usepackage{amsthm}
\usepackage{amsmath}

\journal{Topology and Its Applications}

\begin{document}

\newcommand{\s}{\sigma}
\newcommand{\al}{\alpha}
\newcommand{\om}{\omega}
\newcommand{\be}{\beta}
\newcommand{\la}{\lambda}

\newcommand{\bo}{\mathbf{0}}
\newcommand{\bone}{\mathbf{1}}

\newcommand{\sse}{\subseteq}
\newcommand{\contains}{\supseteq}
\newcommand{\forces}{\Vdash}

\newcommand{\FIN}{\mathrm{FIN}}

\newcommand{\ve}{\vee}
\newcommand{\w}{\wedge}
\newcommand{\bv}{\bigvee}
\newcommand{\bw}{\bigwedge}
\newcommand{\bcup}{\bigcup}
\newcommand{\bcap}{\bigcap}

\newcommand{\rgl}{\rangle}
\newcommand{\lgl}{\langle}
\newcommand{\lr}{\langle\ \rangle}
\newcommand{\re}{\restriction}

\newcommand{\bB}{\mathbb{B}}
\newcommand{\bP}{\mathbb{P}}
\newcommand{\bR}{\mathbb{R}}
\newcommand{\bC}{\mathbb{C}}
\newcommand{\bD}{\mathbb{D}}
\newcommand{\bN}{\mathbb{N}}
\newcommand{\bQ}{\mathbb{Q}}
\newcommand{\bS}{\mathbb{S}}
\newcommand{\St}{\tilde{S}}
   
\newcommand{\sd}{\triangle}
\newcommand{\cl}{\prec}
\newcommand{\cle}{\preccurlyeq}
\newcommand{\cg}{\succ}
\newcommand{\cge}{\succcurlyeq}
\newcommand{\dom}{\mathrm{dom}\,}

\newcommand{\lra}{\leftrightarrow}
\newcommand{\ra}{\rightarrow}
\newcommand{\llra}{\longleftrightarrow}
\newcommand{\Lla}{\Longleftarrow}
\newcommand{\Lra}{\Longrightarrow}
\newcommand{\Llra}{\Longleftrightarrow}
\newcommand{\rla}{\leftrightarrow}
\newcommand{\lora}{\longrightarrow}
\newcommand{\E}{\mathrm{E}}
\newcommand{\rank}{\mathrm{rank}}

\newtheorem{thm}{Theorem}  
\newtheorem{prop}[thm]{Proposition} 
\newtheorem{lem}[thm]{Lemma} 
\newtheorem{cor}[thm]{Corollary} 
\newtheorem{fact}[thm]{Fact}     
\newtheorem*{thmMT}{Main Theorem}
\newtheorem*{thmMTUT}{Main Theorem for $\vec{\mathcal{U}}$-trees}
\newtheorem*{thmnonumber}{Theorem}
\newtheorem*{mainclaim}{Main Claim}

\theoremstyle{definition}   
\newtheorem{defn}[thm]{Definition} 
\newtheorem{example}[thm]{Example} 
\newtheorem{conj}[thm]{Conjecture} 
\newtheorem{prob}[thm]{Problem} 
\newtheorem{examples}[thm]{Examples}
\newtheorem{question}[thm]{Question}
\newtheorem{problem}[thm]{Problem}
\newtheorem{openproblems}[thm]{Open Problems}
\newtheorem{openproblem}[thm]{Open Problem}
\newtheorem{conjecture}[thm]{Conjecture}
\newtheorem*{problem1}{Problem 1}
\newtheorem*{problem2}{Problem 2}
\newtheorem*{problem3}{Problem 3}
\newtheorem*{notn}{Notation}

\theoremstyle{remark} 
\newtheorem*{rem}{Remark} 
\newtheorem*{rems}{Remarks} 
\newtheorem*{ack}{Acknowledgments} 
\newtheorem*{note}{Note}
\newtheorem*{claim}{Claim}
\newtheorem*{claim1}{Claim $1$}
\newtheorem*{claim2}{Claim $2$}
\newtheorem*{claim3}{Claim $3$}
\newtheorem*{claim4}{Claim $4$}
\newtheorem{claimn}{Claim}
\newtheorem{subclaim}{Subclaim}
\newtheorem*{subclaimnn}{Subclaim}
\newtheorem*{subclaim1}{Subclaim (i)}
\newtheorem*{subclaim2}{Subclaim (ii)}
\newtheorem*{subclaim3}{Subclaim (iii)}
\newtheorem*{subclaim4}{Subclaim (iv)}
\newtheorem{case}{Case}

\begin{frontmatter}



\title{Continuous cofinal maps on ultrafilters}


\author{Natasha Dobrinen\corref{cor1}}
\ead{natasha.dobrinen@du.edu}
\ead[url]{http://web.cs.du.edu/$\sim$ndobrine/}
\cortext[cor1]{Corresponding Author}
\address{University of Denver,
Department of Mathematics, 2360 S Gaylord St, Denver, CO 80208, USA, +1.303.871.2120, fax: +1.303.871.3173}

\begin{abstract}
An ultrafilter $\mathcal{U}$ on a countable base {\em has continuous Tukey reductions} if whenever an ultrafilter $\mathcal{V}$ is Tukey reducible to $\mathcal{U}$,
then every monotone cofinal map $f:\mathcal{U}\ra\mathcal{V}$ is continuous  when restricted to some cofinal subset of $\mathcal{U}$.

In the first part of the paper, we give mild conditions under which the property of having  continuous Tukey reductions is inherited under Tukey reducibility.
In particular, if $\mathcal{U}$ is Tukey reducible to a  p-point
then $\mathcal{U}$ has continuous Tukey reductions.
In the second part, we show that any countable iteration of Fubini products of p-points  has  Tukey reductions  which are continuous with respect to its topological Ramsey space of $\vec{\mathcal{U}}$-trees.
\end{abstract}

\begin{keyword}
ultrafilter \sep Tukey \sep cofinal map\sep continuous map \sep p-point
\MSC Primary: 54D80 \sep 03E04; Secondary:  03E05 
\end{keyword}
\end{frontmatter}



\section{Introduction}
\label{sec.intro}

Let $D$ and $E$ be partial orderings.
We say that a function $f:E\ra D$ 
 is {\em cofinal} if the image of each cofinal subset of $E$ is cofinal in $D$.
We say that $D$ is {\em Tukey reducible} to $E$, and write $D\le_T E$, if there is a cofinal map from $E$ to $D$.
An equivalent formulation of Tukey reducibility was noticed by Schmidt in \cite{Schmidt55}.
Given partial orderings $D$ and $E$, a map
$g:D\ra E$ such that the image of each unbounded subset of $D$ is an unbounded subset of $E$ is called a {\em Tukey map} or an {\em unbounded map}.
$D\le_T E$ iff there is a Tukey map from $D$ into $E$.
If both $D\le_T E$ and $E\le_T D$, then we write $D\equiv_T E$ and say that $D$ and $E$ are Tukey equivalent.
$\equiv_T$ is an equivalence relation, and $\le_T$ on the equivalence classes forms a partial ordering. 
The equivalence classes can be called {\em Tukey types}.

The notion of  Tukey reducibility between two  directed partial orderings was first introduced by Tukey in 
 \cite{Tukey40}
to more finely study the Moore-Smith theory of net convergence in topology.
This naturally led to investigations of Tukey types of more general partial orderings, directed and later non-directed. 
These investigations 
 often reveal useful information for the comparison of different partial orderings.
For example, Tukey reducibility downward preserves calibre-like properties, such as the countable chain condition, property K, precalibre $\aleph_1$, $\sigma$-linked, and $\sigma$-centered (see \cite{Todorcevic96}).
For more on classification theories of Tukey types for certain classes of ordered sets, we refer the reader to \cite{Tukey40},
\cite{Day44}, \cite{Isbell65}, \cite{TodorcevicDirSets85}, and
\cite{Todorcevic96}.


In this paper we continue a recent line of research into the structure of the Tukey types of ultrafilters on $\om$ ordered by reverse inclusion.
(See \cite{Milovich08}, \cite{Dobrinen/Todorcevic10}, 
 \cite{Raghavan/Todorcevic11}, and \cite{Dobrinen/Todorcevic11}.)
For any ultrafilter $\mathcal{U}$  on $\om$,
$(\mathcal{U},\contains)$ is a directed partial ordering.
We remark  that 
for any two directed partial orderings $D$ and $E$, $D\equiv_T E$ iff $D$ and $E$ are {\em cofinally similar}; that is, there is a partial ordering  into which both $D$ and $E$ embed as cofinal subsets (see \cite{Tukey40}).
So for ultrafilters, Tukey equivalence is the same as cofinal similarity.

For ultrafilters, we may restrict our attention to monontone cofinal maps.
We say that a map $f:\mathcal{U}\ra\mathcal{V}$ is \em monotone \rm if
 for any $X,Y\in\mathcal{U}$, $X\contains Y$ implies $f(X)\contains f(Y)$.
It is not hard to show that whenever $\mathcal{U}\ge_T\mathcal{V}$, then there is a \em monotone \rm cofinal map witnessing this (see Fact 6 of \cite{Dobrinen/Todorcevic10}).
Thus, we shall generally assume that each cofinal map under consideration is monotone.

Another motivation for 
this study  is that Tukey reducibility is a 
generalization of  Rudin-Keisler reducibility.
Recall that 
$\mathcal{U}\ge_{RK}\mathcal{V}$ iff there is a function $f:\om\ra\om$ such that the ultrafilter generated by the collection $\{f(U):U\in\mathcal{U}\}$
is equal to $\mathcal{V}$.
Whenever $\mathcal{U}\ge_{RK}\mathcal{V}$, then also $\mathcal{U}\ge_{T}\mathcal{V}$ (see Fact 1 in \cite{Dobrinen/Todorcevic10}).
In general, Tukey and Rudin-Keisler reducibility are  quite distinct.
Various instances of this can be seen in \cite{Dobrinen/Todorcevic10} and \cite{Raghavan/Todorcevic11}, and in the following.

\begin{thm}[Isbell \cite{Isbell65}]\label{thm.3}
There is an ultrafilter $\mathcal{U}_{\mathrm{top}}$ on $\om$ realizing the maximal cofinal type among all directed sets of cardinality continuum; that is, $\mathcal{U}_{\mathrm{top}}\equiv_T[\mathfrak{c}]^{<\om}$.
\end{thm}

\begin{rem}
The same construction in Isbell's  proof was done independently by  Juh\'{a}sz in \cite{Juhasz66} (stated in \cite{Juhasz67}) in connection with strengthening a theorem of Posp{\'{i}}{\v{s}}il \cite{Pospisil39}, though without the Tukey terminology.
\end{rem}

Note that there are  $2^{\mathfrak{c}}$ many different ultrafilters of maximal Tukey type, since any collection of independent sets can be used in a canonical way to construct an ultrafilter with maximal type.
Thus the top Tukey type has cardinality $2^{\mathfrak{c}}$.
In contrast, every Rudin-Keisler equivalence class has cardinality $\mathfrak{c}$.
Moreover, there is no maximal equivalence class in the Rudin-Keisler ordering.
So the maximal Tukey class is contains $2^{\mathfrak{c}}$ many Rudin-Keisler equivalence classes, none of which is maximal in the Rudin-Keisler sense.
Thus,  for the case of the maximal Tukey type, the Rudin-Keisler equivalence relation is strictly finer than the Tukey 
equivalence relation.

We now turn our attention to p-points.
\begin{defn}\label{defn.p-point}
An ultrafilter $\mathcal{U}$ on $\om$ is a \em p-point \rm iff for each decreasing sequence $A_0\contains A_1\contains \dots$ of elements of $\mathcal{U}$, there is an $A\in\mathcal{U}$
such that $A\sse^* A_n$, for all $n<\om$.
\end{defn}

We note that Isbell's Problem \cite{Isbell65}, whether it is true in ZFC that there is an ultrafilter with Tukey type strictly below the maximal type, is still open.
It was shown in \cite{Dobrinen/Todorcevic10} that countable iterations of Fubini products of p-points (and in fact the more general class of so-called ``basically generated'' ultrafilters) are strictly below the maximal Tukey type.

It follows from work in \cite{Solecki/Todorcevic04} that p-points 
have the following special property:
If $\mathcal{U}$ is a p-point and $\mathcal{V}\le_T\mathcal{U}$, 
then there is a 
 definable monotone cofinal map from $\mathcal{U}$ into $\mathcal{V}$. 
Hence every p-point has Tukey type of cardinality $\mathfrak{c}$.
This fact is quite useful in analyzing  the structure of the Tukey types of p-points, as seen in \cite{Dobrinen/Todorcevic10} and \cite{Raghavan/Todorcevic11}.

In fact, p-points have even stronger properties in terms of cofinal maps.
Identify $\mathcal{P}(\om)$ with $2^{\om}$, the set of characteristic functions of subsets of  $\omega$, and endow $\mathcal{P}(\om)$ with the corresponding topology. 
A sequence $(X_n)_{n<\om}$ of elements of $\mathcal{P}(\om)$ converges to an element $X\in\mathcal{P}(\om)$ iff for each $k<\om$ there is an $N<\om$ such that for each $n\ge N$, $X_n\cap k = X\cap k$.
A function $f:\mathcal{P}(\om)\ra\mathcal{P}(\om)$ is \em continuous \rm iff whenever $X_n\ra X$, then also $f(X_n)\ra f(X)$.
A function $f:\mathcal{U}\ra\mathcal{V}$ is said to be  continuous if it is continuous on $\mathcal{U}$ considered as a topological subspace of $\mathcal{P}(\om)$.
\begin{defn}
(1)
An ultrafilter $\mathcal{U}$ \em supports continuous cofinal maps \rm if whenever $\mathcal{V}\le_T\mathcal{U}$, there is a continuous, monotone cofinal map $f:\mathcal{U}\ra\mathcal{V}$.

(2)
$\mathcal{U}$ \em
has continuous Tukey reductions \rm if whenever $f:\mathcal{U}\ra\mathcal{V}$ is a monotone cofinal map, 
then there is a cofinal subset $\mathcal{X}\sse\mathcal{U}$ such that $f\re\mathcal{X}$ is continuous, with respect to the subspace topology on $\mathcal{X}$ inherited from $2^{\om}$.

(3)
$\mathcal{U}$ \em has basic Tukey reductions \rm if whenever $f:\mathcal{U}\ra\mathcal{V}$ is a monotone cofinal map, there is a cofinal subset $\mathcal{X}\sse\mathcal{U}$ such that $f\re\mathcal{X}$ is continuous; moreover, there is a continuous monotone map $\tilde{f}:\mathcal{P}(\om)\ra\mathcal{P}(\om)$ such that $\tilde{f}\re\mathcal{X}=f\re\mathcal{X}$, and $\tilde{f}\re \mathcal{U}:\mathcal{U}\ra\mathcal{V}$ is a cofinal map.
\end{defn}
Note that (3) $\Rightarrow$ (2) $\Rightarrow$ (1). 
We point out  that any ultrafilter which supports continuous cofinal maps has Tukey type of cardinality $\mathfrak{c}$.
The following theorem  served as motivation for the study of which ultrafilters have continuous Tukey reductions, and when, if ever, the property of having continuous Tukey reductions is inherited under Tukey reducibility.
The following is a stronger restatement of  Theorem 20 of Dobrinen and Todorcevic in \cite{Dobrinen/Todorcevic10}, which follows from the proof.

\begin{thm}[Dobrinen/Todorcevic \cite{Dobrinen/Todorcevic10}]\label{thm.ppoint}
Suppose $\mathcal{U}$ is a p-point on $\om$.
Then $\mathcal{U}$ has basic Tukey reductions.
\end{thm}

The existence of continuous cofinal maps is useful for analyzing properties of p-points added by $\sigma$-closed forcings.
They are also used in 
the following theorem, which reveals the surprising fact that the Tukey and Rudin-Keisler orders sometimes coincide.
Recall that $\le_{RB}$ is the Rudin-Blass ordering, which implies $\le_{RK}$.

\begin{thm}[Raghavan \cite{Raghavan/Todorcevic11}]\label{thm.RT}
Let $\mathcal{U}$ be any ultrafilter and let $\mathcal{V}$ be a q-point.
Suppose $f:\mathcal{U}\ra\mathcal{V}$ is continuous, monotone, and cofinal in $\mathcal{V}$.
Then  $\mathcal{V}\le_{RB}\mathcal{U}$.
\end{thm}

The purpose of this paper is to show that the property of having basic Tukey reductions is inherited under Tukey reducibility, and that, assuming the existence of p-points,
 there is a wide class of ultrafilters, and more general $\vec{\mathcal{U}}$-tree spaces, which have basic Tukey reductions.

\begin{thmMT}\label{thm.main}
If $\mathcal{U}$ is Tukey reducible to a  p-point (or a stable-ordered union ultrafilter),
then $\mathcal{U}$ has basic Tukey reductions.
\end{thmMT}

We remark that the property of having basic Tukey reductions is the only property yet known to be inherited under Tukey reducibility, whereas many standard properties, such as being a  p-point or selective, are inherited under Rudin-Keisler reducibility but not under Tukey reducibility.

In the next section we will prove the Main Theorem.  
In Section \ref{sec.Fubit}, we
introduce the notion of $\vec{\mathcal{U}}$-trees, and 
alert the reader to the connection between countable iterations of Fubini products of ultrafilters and  corresponding  $\vec{\mathcal{U}}$-trees.
Then
we show that every countable iteration of Fubini products of p-points has basic Tukey reductions in the topological space of  $\vec{\mathcal{U}}$-trees. 
That countable iterations of Fubini products of p-points (or stable ordered-union ultrafilters on $\FIN$) come into play is not surprising.
It is a theorem of Todorcevic in \cite{Raghavan/Todorcevic11} that whenever $\mathcal{U}$ is selective and $\mathcal{V}\le_T\mathcal{U}$, then $\mathcal{V}$ is Rudin-Keisler equivalent to a countable iteration of Fubini products of $\mathcal{U}$.
Recently, similar results  were found to hold for 
weakly Ramsey ultrafilters and the more general class of  ultrafilters $\mathcal{U}_{\al}$ ($\al<\om_1$) introduced and investigated by Laflamme in \cite{Laflamme89}.
 (See the forthcoming \cite{Dobrinen/Todorcevic11} for more details.)
\vskip.1in

\it Acknowledgments. 
\rm
We  thank Stevo Todorcevic for generous sharing of his knowledge, without which this paper would not exist.
We thank the referee for insightful comments, which have improved the exposition of this paper.


\section{Basic Tukey reductions are preserved under Tukey reducibility}\label{sec.pres}

In this section we present a sufficient condition which  guarantees that the property of having 
continuous cofinal maps is inherited under Tukey reducibility.
The condition is that of having basic Tukey reductions.
As a particular consequence, we obtain the Main Theorem.

We use $2^{<\om}$ to denote the collection of finite sequences  $s:n\ra 2$, for $n<\om$. 
For $s,t\in 2^{<\om}$,
we write $s\sqsubseteq t$ to denote that $\dom(s)\sse\dom(t)$ and that 
$t\re\dom(s)=s$;
in other words, $s$ is an initial segment of $t$.
We also use $a\sqsubseteq X$ for sets $a,X\sse\om$ to denote that, given their strictly increasing enumerations, $a$  is an initial segment of $X$.

We would like to think of  $s$ as identified with the set for which it is the characteristic function. 
Of course, since one set  determines  different characteristic functions on different domains, we take the slightly tedious but unambiguous road of differentiating between a set and its characteristic function on a given domain.
Thus,
for $s\in 2^m$, $s^{-1}(\{1\})$ is the subset of $m$ for which $s$ is the characteristic function.
Since this notation would be cumbersome, 
we shall instead use $d(s)$ to denote $s^{-1}(\{1\})$.
We think of $d(s)$ as the portion of the {\em domain of $s$ which gets mapped to $1$}, which is exactly the set for which $s$ is a characteristic function.

Let $D$ be a subset of $2^{<\om}$.
We shall call a map $\hat{f}:D\ra 2^{<\om}$ \em level preserving \rm if 
there is a strictly increasing sequence  $(k_m)_{m<\om}$ such that 
$\{k_m:m<\om\}=\{|s|:s\in D\}$ 
and
for each  $s\in D\cap 2^{k_m}$, we have that $\hat{f}(s)\in 2^m$.
A level preserving map $\hat{f}$ is \em initial segment preserving \rm if for each $m<m'$, each $s\in D\cap 2^{k_m}$ and each $s'\in D\cap 2^{k_{m'}}$ such that $s\sqsubseteq s'$, then also $\hat{f}(s)\sqsubseteq\hat{f}(s')$.
A map $\hat{f}:2^{<\om}\ra 2^{<\om}$ is 
{\em monotone} if for each $s,t\in 2^{<\om}$ such that $d(s)\sse d(t)$, we have $d(\hat{f}(s))\sse d(\hat{f}(t))$.

\begin{defn}\label{defn.basic}
A monotone map $f$ on a subset $\mathcal{D}\sse \mathcal{P}(\om)$ is said to be {\em basic} iff $f$ is 
generated  by a monotone, level and initial segment preserving map.
This means that 
 there is some  strictly
increasing sequence $(k_m)_{m<\om}$ such that,
letting 
\begin{equation}
D=\bigcup_{m<\om}\{s\in 2^{k_m}: \exists X\in \mathcal{D}\, (d(s)=X\cap k_m)\},
\end{equation}
there is a 
 level and initial segment preserving map
$\hat{f}:D\ra 2^{<\om}$ 
such that 
for each $X\in\mathcal{D}$,
\begin{equation} 
f(X)=\bigcup\{d(\hat{f}(s)):s\in D \mathrm{\ and\ } d(s)\sqsubseteq X\}.
\end{equation}
In this case, we will also say that $\hat{f}$ \em generates \rm $f$.
\end{defn}

Note that 
 $\hat{f}$ being initial segment preserving implies that
\begin{equation}
\mathrm{ for\ each\ }m<m' \mathrm{\ and\ }s\in D\cap 2^{k_{m'}},\ 
 \hat{f}(s\re k_m)=\hat{f}(s)\re m.
\end{equation}
Thus, a basic map $f$ on $\mathcal{D}$ is  continuous on $\mathcal{D}$, and furthermore, generates a continuous map on $\overline{\mathcal{D}}$, the closure of $\mathcal{D}$ in $\mathcal{P}(\om)$.

We begin with a theorem showing that any basic cofinal map from some cofinal subset of an ultrafilter $\mathcal{U}$  into another ultrafilter $\mathcal{V}$  can be extended to a basic map on the whole space $\mathcal{P}(\om)$ in such a way that its restriction to $\mathcal{U}$ is a continuous cofinal map.

\begin{thm}\label{thm.PropACtsMaps}
Suppose $\mathcal{U}$ and $\mathcal{V}$ are  ultrafilters,
$f:\mathcal{U}\ra\mathcal{V}$
is a  monotone cofinal map, and
there is a cofinal subset $\mathcal{D}\sse\mathcal{U}$ such that
  $f\re\mathcal{D}$ is basic.
Then there is a continuous, monotone $\tilde{f}:\mathcal{P}(\om)\ra\mathcal{P}(\om)$ 
such that
\begin{enumerate}
\item
 $\tilde{f}$
is basic;
\item
 $\tilde{f}\re\mathcal{D}= f\re\mathcal{D}$; and
 \item $\tilde{f}\re\mathcal{U}:\mathcal{U}\ra\mathcal{V}$ is a cofinal map.
\end{enumerate}
Thus, if  $\mathcal{U}$ has the property that for every $\mathcal{V}\le_T\mathcal{U}$ and every monotone cofinal map $f:\mathcal{U}\ra\mathcal{V}$ there is some cofinal $\mathcal{D}\sse\mathcal{U}$ for which $f\re\mathcal{D}$ is basic,
then $\mathcal{U}$ has basic Tukey reductions.
\end{thm}

\begin{proof}
We first extend
 the basic map  $f\re\mathcal{D}$ 
 to a map on all of $\mathcal{U}$.
Define $f'$ on $\mathcal{U}$ by 
\begin{equation}
f'(U)=\bigcup\{f(X):X\in\mathcal{D}\mathrm{\ and\ } X\sse U\},
\end{equation}
for each $U\in\mathcal{U}$.

\begin{claimn}\label{claim.A1}
$f'$ is a monotone cofinal map from $\mathcal{U}$ into $\mathcal{V}$, and 
$f'\re\mathcal{D}=f\re\mathcal{D}$.
\end{claimn}

\begin{proof}
Let $U\in\mathcal{U}$.
Then $f'(U)$ is a union of elements in $\mathcal{V}$, hence is itself in $\mathcal{V}$.
By its definition, it is easy to see that $f'$ is monotone.
Given any $X\in\mathcal{D}$, by definition $f'(X)\contains f(X)$.
Since $f$ is monotone, for each $X'\in \mathcal{D}$ such that $X'\sse X$, $f(X')\sse f(X)$. 
Thus, $f'(X)\sse f(X)$.
Hence, $f'(X)=f(X)$ for all $X\in\mathcal{D}$.
Moreover, the image of $\mathcal{U}$ under $f'$ contains the image of $\mathcal{D}$ under $f$, which is cofinal in $\mathcal{V}$.
Hence, $f'$ is a monotone cofinal map from $\mathcal{U}$ into $\mathcal{V}$.
\end{proof}

Let $\hat{f}$ be an initial segment and level preserving map which generates $f\re\mathcal{D}$.
Thus, there is a strictly increasing sequence $(k_m)_{m<\om}$ such that,
letting 
$D=\bigcup_{m<\om}\{s\in 2^{k_m}: \exists X\in \mathcal{D}\, (d(s)=X\cap k_m)\}$,
$\hat{f}:D\ra 2^{<\om}$ and 
 for each $s\in D\cap 2^{k_m}$,
 $\hat{f}(s)\in 2^m$.
Then  
  for each $m<m'$ and $s\in D\cap 2^{k_{m'}}$,
 $\hat{f}(s\re k_m)=\hat{f}(s)\re m$; and 
 for each $X\in\mathcal{D}$, 
\begin{equation} 
f(X)=\bigcup\{d(\hat{f}(s)):s\in D\mathrm{\ and\ }d(s)\sqsubseteq X\}.
\end{equation}

\begin{claimn}\label{claim.A2}
There is a monotone, level and  initial segment preserving map $\hat{g}$ which  generates a function $\tilde{f}:\mathcal{P}(\om)\ra\mathcal{P}(\om)$ such that  $\tilde{f}\re\mathcal{U}=f'$.
\end{claimn}

\begin{proof}
Note that for each $m<\om$,
\begin{equation}
\{t\in 2^{k_m}:\exists s\in D\cap 2^{k_m}(d(s)\sse d(t))\}=2^{k_m},
\end{equation}
since every finite sequence of zeros of length $2^{k_m}$ for some $m$
is in $D$ since $\mathcal{D}$ is cofinal in $\mathcal{U}$.
Let $C=\bigcup_{m<\om}2^{k_m}$.
Define $\hat{g}$ on $C$ as follows.
For $t\in 2^{k_m}$, define
$\hat{g}(t)$ to be the characteristic function with domain $m$ of the set 
\begin{equation}
\bigcup\{d(\hat{f}(s)):s\in D\cap \bigcup_{n\le m}2^{k_n}\mathrm{\ and\ }d(s)\sse d(t)\}.
\end{equation}
Abusing terminology and confusing sets with their characteristic functions in this sentence,
the idea behind $\hat{g}(t)$ is that it codes 
 the union of all  sets 
which  are $\hat{f}$-images 
 of sets
 in the domain of $\hat{f}$
  contained within $t$.
By its definition, $\hat{g}$ is monotone and takes each $k_m$-th level to the $m$-th level.
So $\hat{g}$ is level preserving.

To see that $\hat{g}$ is initial segment preserving, let $t\sqsubset t'$, where $t\in 2^{k_m}$ and $t'\in 2^{k_{m'}}$ for some $m<m'$.
For each $s\in D\cap \bigcup_{n\le m'}2^{k_n}$ with $d(s)\sse d(t')$, 
we have that $d(s\re k_m)\sse d(t)$.
So
\begin{align}
d(\hat{g}(t')\re m)
&=
\bigcup\{d(\hat{f}(s)):s\in D\cap \bigcup_{n\le 
m'}2^{k_n}\mathrm{\ and\  }s\sse s'\}\cap m\\
&=
\bigcup\{d(\hat{f}(s)\re m): s\in D\cap \bigcup_{n\le m'}2^{k_n}\mathrm{\ and\ }d(s)\sse d(t')\}\\
&=
\bigcup\{d(\hat{f}(s\re k_m)): s\in D\cap \bigcup_{n\le m'}2^{k_n}\mathrm{\ and\ }d(s)\sse d(t')\}\\
&=d(\hat{g}(t)).
\end{align}
Since $m=\dom(\hat{g}(t))$, we have that 
$\hat{g}(t')\re m=\hat{g}(t)$.
Therefore, $\hat{g}(t)\sqsubset \hat{g}(t')$.

Now define $\tilde{f}:\mathcal{P}(\om)\ra\mathcal{P}(\om)$ by
\begin{equation}
\tilde{f}(Z)
=\bigcup\{d(\hat{g}(t)):\exists m<\om\, ( t\in 2^{k_m}\mathrm{\ and \ } d(t)\sqsubseteq Z)\}.
\end{equation}
Then, by its definition,
$\tilde{f}$ is generated by the level and initial segment preserving map $\hat{g}$.
Note that $\tilde{f}$ is monotone and continuous just by virtue of its definition,
and that 
\begin{equation}
\tilde{f}(Z)=\bigcup\{d(\hat{g}(t)):\exists m<\om\, ( t\in 2^{k_m}\mathrm{\ and\ }d(t)\sse Z)\},
\end{equation}
 since $\hat{g}$ is monotone.

Lastly, we check that $\tilde{f}\re\mathcal{U}= f'$.
Let $U\in\mathcal{U}$.
\begin{align}
f'(U)
&=
\bigcup\{d(\hat{f}(s)):\exists m<\om\, (s\in 2^{k_m}\mathrm{\ and \ }\exists X\in\mathcal{D}(d(s)\sqsubseteq X\sse U))\}\\
&=
\bigcup\{d(\hat{f}(s)):s\in D\mathrm{\ and\ }d(s)\sse U\}.
\end{align}
At the same time,
\begin{align} 
\tilde{f}(U)
&=\bigcup\{d(\hat{g}(t)):\exists m<\om\, (t\in 2^{k_m}\mathrm{\ and \ }d(t)\sse U)\}\\
&=
\bigcup\{d(\hat{f}(s)):s\in D\mathrm{\ and\ }d(s)\sse U\}.
\end{align}
Therefore, $\tilde{f}(U)=f'(U)$. 
\end{proof}

Since $\tilde{f}\re\mathcal{U}=f'$,
we have the conclusion of the theorem.
\end{proof}

In the next theorem, we give a condition under which the property of supporting continuous cofinal maps gets inherited under Tukey reducibility.

\begin{thm}\label{thm.PropB}
Let $\mathcal{U}$ be an ultrafilter such that whenever $f:\mathcal{U}\ra\mathcal{V}$ is a monotone cofinal function,
then there exists a cofinal $\mathcal{D}\sse\mathcal{U}$ such that 
$f\re\mathcal{D}$ is basic.
Then for every ultrafilter $\mathcal{W}\le_T\mathcal{U}$,
$\mathcal{W}$ has basic Tukey reductions.
\end{thm}

\begin{proof}
Let $\mathcal{U}$ be as in the hypotheses and $\mathcal{W}\le_T\mathcal{U}$.
Let  $h:\mathcal{W}\ra\mathcal{V}$ be a monotone cofinal map.
Extend $h$ to the  monontone map $\tilde{h}:\mathcal{P}(\om)\ra\mathcal{P}(\om)$  defined as follows.
For each $X\in\mathcal{P}(\om)$, let
\begin{equation}
\tilde{h}(X)=\bigcap\{h(W):W\in\mathcal{W}\mathrm{\ and\ }W\contains X\}.
\end{equation}
Note that $\tilde{h}$ is monotone and $\tilde{h}\re\mathcal{W}=h$.
By the hypotheses and Theorem \ref{thm.PropACtsMaps}, there is a continuous monotone map $\tilde{f}:\mathcal{P}(\om)\ra\mathcal{P}(\om)$ such that 
letting $f=\tilde{f}\re\mathcal{U}$,
$f:\mathcal{U}\ra\mathcal{W}$ is a cofinal map,
 and $\tilde{f}$ 
is basic, generated by a monotone,
 level and initial segment preserving map $\hat{f}:\bigcup_{m<\om}2^{k_m}\ra 2^{<\om}$, for some strictly increasing sequence $(k_m)_{m<\om}$.

Let $\tilde{g}=\tilde{h}\circ \tilde{f}$.
Then $\tilde{g}:\mathcal{P}(\om)\ra\mathcal{P}(\om)$ and is monotone.
Letting $g=\tilde{g}\re \mathcal{U}$,
we see that
$g=h\circ f$, hence
 $g:\mathcal{U}\ra\mathcal{V}$ is a monotone cofinal map.
By the hypotheses, 
there is a cofinal subset $\mathcal{D}\sse\mathcal{U}$ 
such that 
$g\re\mathcal{D}$ is 
basic, generated by some monotone, level and initial segment preserving map $\hat{g}$.
Without loss of generality, we may assume that $\hat{f}$ and $\hat{g}$ are defined on the same set of levels $\bigcup_{m<\om}2^{k_m}$:
For if $\hat{g}$ is defined on $\bigcup_{m<\om}2^{j_m}$, take $k'_m=\max(k_m,j_m)$, and for $s\in 2^{k'_m}$, define $\hat{f}'(s)=\hat{f}(s\cap k_m)$ and $\hat{g}'(s)=\hat{g}(s\cap j_m)$.
Let 
\begin{equation}
D=\bigcup_{m<\om}\{s\in 2^{k_m}:\exists X\in\mathcal{D}(d(s)=X\cap k_m)\}.
\end{equation}
Note that for each 
$s\in D\cap 2^{k_m}$ such that $d(s)\sqsubseteq X\in\mathcal{D}$,
we have that $\hat{f}(s)=f(X)\cap m$ 
and $\hat{g}(s)=g(X)\cap m =\tilde{g}(X)\cap m$.

Let $\mathcal{Y}=f''\mathcal{D}$.  Then $\mathcal{Y}$ is cofinal in $\mathcal{W}$.
Let $\overline{\mathcal{D}}$ denote the closure of $\mathcal{D}$ in the topological space $\mathcal{P}(\om)$.
Since $f$ is continuous on the compact space $\mathcal{P}(\om)$,
$\overline{\mathcal{Y}}=\overline{f''\mathcal{D}}=f''\overline{\mathcal{D}}$.
Let $C$ be the collection of all characteristic functions of initial segments of elements of $\mathcal{Y}$.
That is, 
\begin{equation}
C=\bigcup_{m<\om}\{s\in 2^m:\exists Y\in\mathcal{Y}\, (d(s)=Y\cap m)\}.
\end{equation}
We point out that $C$ is also the collection of all
characteristic functions of
 initial segments of elements of $\overline{\mathcal{Y}}$. 
Define $\hat{h}:C\ra 2^{<\om}$ as follows: 
For  $t\in C\cap 2^{m}$, 
define $\hat{h}(t)$
to be the characteristic function with domain $m$ such that 
\begin{equation}
d(\hat{h}(t))=
\bigcap\{d(\hat{g}(s)):s\in D\cap 2^{k_m}\mathrm{\  and\ } \hat{f}(s)=t\}.
\end{equation}
Thus, $d(\hat{h}(t))\sse h(X)\cap m$.
Note that $\hat{h}$ is level preserving, just by its definition.
In fact, for each $t\in C\cap 2^m$, $\hat{h}(t)$ is also in $2^m$.

The problem with  $\hat{h}$ is that it is, a priori, neither monotone nor initial segment preserving, and it is not initially clear whether or not $\hat{h}$ generates $h$.
We will show, however, that when restricted to a certain set of levels and truncated by a certain amount, $\hat{h}$ is in fact monotone and initial segment preserving.
Towards showing such levels exist, we have the next claim.

\begin{claim1}\label{claim.Bii}
Let $Y\in\overline{\mathcal{Y}}$.
For each $\tilde{m}$ there is an $m\ge\tilde{m}$ such that for each $Z\in\overline{\mathcal{D}}$ with $\tilde{f}(Z)\cap m=Y\cap m$ there is an $X\in\overline{\mathcal{D}}$ such that $\tilde{f}(X)=Y$ and $\tilde{g}(X)\cap \tilde{m}=\tilde{g}(Z)\cap \tilde{m}$.
\end{claim1}

\begin{proof}
Let $Y\in\overline{\mathcal{Y}}$.
Suppose the claim fails.
Then
there is an $\tilde{m}$ such that for each $m\ge\tilde{m}$,
there is a $Z_m\in\overline{\mathcal{D}}$ such that $\tilde{f}(Z_m)\cap m=Y\cap m$,
but for each $X\in\overline{\mathcal{D}}$ such that $\tilde{f}(X)=Y$,
$\tilde{g}(X)\cap \tilde{m}\ne \tilde{g}(Z_m)\cap \tilde{m}$.
Since $\overline{\mathcal{D}}$ is a closed subset of $\mathcal{P}(\om)$, it is compact,
so there is a convergent subsequence $(Z_{m_i})_{i<\om}$ which converges to some $Z\in\overline{\mathcal{D}}$.
Since $\tilde{f}$ is continuous, $\tilde{f}(Z_{m_i})$ converges to $\tilde{f}(Z)$.
Since   $\tilde{f}(Z_{m_i})\cap m_i=Y\cap m_i$ for each $i$,
it follows that $\tilde{f}(Z_{m_i})$ converges to $Y$.
Therefore, $\tilde{f}(Z)=Y$.
Since $\tilde{g}$ is continuous, $\tilde{g}(Z_{m_i})$ converges to $\tilde{g}(Z)$.
But that implies that for all sufficiently large values of $i$,
$\tilde{g}(Z_{m_i})\cap \tilde{m}=\tilde{g}(Z)\cap \tilde{m}$,
contradicting that for all $m$, $\tilde{g}(Z_m)\cap \tilde{m}\ne \tilde{g}(Z)\cap \tilde{m}$.
\end{proof}

\begin{claim2}
There is a strictly increasing sequence 
$(j_{\tilde{m}})_{\tilde{m}<\om}$  such that
for each $Y\in\overline{\mathcal{Y}}$,
 $\tilde{m}<\om$, and
 $Z\in\overline{\mathcal{D}}$ with $\tilde{f}(Z)\cap j_{\tilde{m}}=Y\cap j_{\tilde{m}}$, 
there is an $X\in\overline{\mathcal{D}}$ such that $\tilde{f}(X)=Y$ and $\tilde{g}(X)\cap \tilde{m}=\tilde{g}(Z)\cap \tilde{m}$.
\end{claim2}

\begin{proof} 
For each $\tilde{m}$ and  $Y\in\overline{\mathcal{Y}}$, there is an $m(Y,\tilde{m})$ satisfying Claim 1.
The finite segments $Y\cap m(Y,\tilde{m})$ determine open sets, and the union of these open sets (over all $Y\in\overline{\mathcal{Y}}$)
 covers $\overline{\mathcal{Y}}$.
Thus, there is a finite subcover, say $Y_0\cap m(Y_0,\tilde{m}),\dots, Y_n\cap m(Y_n,\tilde{m})$.
Take $j_{\tilde{m}}\ge\max\{\tilde{m},m(Y_0,\tilde{m}),\dots,
m(Y_n,\tilde{m})\}$ so that  $(j_{\tilde{m}})_{\tilde{m}<\om}$  forms a strictly increasing sequence.
This sequence satisfies the claim.
\end{proof}

\begin{claim3}\label{claim.FactConverge}
Let $Y\in\overline{\mathcal{Y}}$ and $\tilde{m}$ be given,
 and let $t$ be the characteristic function of $Y\cap j_{\tilde{m}}$ with domain $j_{\tilde{m}}$.
Then $\tilde{h}(Y)\cap\tilde{m}=d(\hat{h}(t))\cap\tilde{m}$.
\end{claim3}

\begin{proof}
Let $Y\in\overline{\mathcal{Y}}$ and $\tilde{m}$ be given,
and let $t$ be the characteristic function of $Y\cap j_{\tilde{m}}$ with domain $j_{\tilde{m}}$.
By definition of $\hat{h}$,
\begin{equation}
d(\hat{h}(t))
=\bigcap\{d(\hat{g}(s)):s\in D\cap 2^{k_{j_{\tilde{m}}}}
\mathrm{\ and\ }\hat{f}(s)= t\}.
\end{equation}
Let $s\in D\cap 2^{k_{j_{\tilde{m}}}}$ such that $\hat{f}(s)=t$.
By Claim 2,
 there is an $X\in\overline{\mathcal{D}}$ such that $d(s)=X\cap k_{j_{\tilde{m}}}$ and $\tilde{f}(X)=Y$.
Thus,
$\tilde{h}(Y)=\tilde{g}(X)$, and
\begin{equation}
\tilde{h}(Y)\cap\tilde{m}
=
\tilde{g}(X)\cap\tilde{m}
=
d(\hat{g}(s\re k_{\tilde{m}}))
=
d(\hat{g}(s))\cap\tilde{m}.
\end{equation}
Thus, $\tilde{h}(Y)\cap\tilde{m}=d(\hat{h}(t))\cap\tilde{m}$.
\end{proof}

It follows from Claim 3 that,
 for each $Y\in\overline{\mathcal{Y}}$,
 letting $s_{\tilde{m}}$ be the characteristic function for
$Y\cap j_{\tilde{m}}$,
\begin{equation} 
\tilde{h}(Y)=\bigcup_{\tilde{m}<\om}\hat{h}(s_{\tilde{m}})\cap\tilde{m}.
\end{equation}

Finally, we define $\hat{i}$ on domain $C\cap \bigcup_{\tilde{m}<\om}2^{j_{\tilde{m}}}$
as follows.
For each $\tilde{m}<\om$ and each $t\in C\cap 2^{j_{\tilde{m}}}$,
define
\begin{equation}
\hat{i}(t)= \hat{h}(t)\re\tilde{m}.
\end{equation}

\begin{claim4}
$\hat{i}$ is a monotone, level and initial segment preserving map which generates $\tilde{h}\re\overline{\mathcal{Y}}$, and hence generates $h\re\mathcal{Y}$.
\end{claim4}

\begin{proof}
By its definition, $\hat{i}$ is level preserving.
It is monotone and initial segment preserving, since $\hat{h}$ is monontone and initial segment preserving.
Let $Y\in\overline{\mathcal{Y}}$, and for each $\tilde{m}$, let $t_{\tilde{m}}$ be the characteristic function of $Y\cap j_{\tilde{m}}$.
It follows from Claim 3 and the definition of $\hat{i}$ that 
\begin{equation}
\tilde{h}(Y)
=\bigcup_{\tilde{m}<\om}d(\hat{h}(t_{\tilde{m}}))\cap\tilde{m}
=\bigcup_{\tilde{m}<\om}d(\hat{i}(t_{\tilde{m}})).
\end{equation}
Thus, $\hat{i}$ generates $\tilde{h}\re\overline{\mathcal{Y}}$.
\end{proof}

Thus, $h\re\mathcal{Y}=\tilde{h}\re\mathcal{Y}$ is basic, generated by $\hat{i}$.
By Theorem \ref{thm.PropACtsMaps},
$h\re\mathcal{Y}$ extends to some basic map $h':\mathcal{P}(\om)\ra\mathcal{P}(\om)$ such that $h'\re\mathcal{W}\ra\mathcal{U}$ is cofinal.
\end{proof}

\begin{rem}
Every p-point satisfies the conditions of Theorem \ref{thm.PropB}
as was shown in the proof of Theorem 20 of \cite{Dobrinen/Todorcevic10}, where the cofinal set $\mathcal{D}$ there is of the simple form $\mathcal{P}(\tilde{X})\cap\mathcal{U}$ for some particular $\tilde{X}\in\mathcal{U}$.

There is a notion of ultrafilter on the base $\FIN=[\om]^{<\om}\setminus\{\emptyset\}$ 
called stable ordered-union ultrafilter, which is the analogue of p-point for ultrafilters on the base set $\FIN$.
In 
 Theorems 71 and 72 of \cite{Dobrinen/Todorcevic10},
 it was shown that for each stable ordered union ultrafilter $\mathcal{U}$, both $\mathcal{U}$ and its projection $\mathcal{U}_{\min,\max}$ support basic cofinal maps on some cofinal subset.
 Thus, by Theorem \ref{thm.PropB},
 every ultrafilter Tukey below these kinds of ultrafilters also has basic Tukey reductions.
 It is perhaps of interest that the $\mathcal{U}_{\min,\max}$ is rapid, but neither a p-point nor a q-point.
Rather than add many definitions here, we refer the interested reader to \cite{Blass87} and  
 \cite{Dobrinen/Todorcevic10}.
\end{rem}

The Main Theorem follows from Theorems \ref{thm.PropACtsMaps} and \ref{thm.PropB}, along with Theorems 20, 71, and 72 from  \cite{Dobrinen/Todorcevic10}.

\begin{thmMT}\label{thm.main}
If $\mathcal{U}$ is Tukey reducible to a  p-point (or a stable-ordered union ultrafilter),
then $\mathcal{U}$ has basic Tukey reductions.
\end{thmMT}


\section{$\vec{\mathcal{U}}$-trees of p-points have basic Tukey reductions}\label{sec.Fubit}

In this section, we alert the reader to  the connection between countable iterations of  Fubini products of ultrafilters and ultrafilters of $\vec{\mathcal{U}}$-trees.
We then show that given  any countable iteration of Fubini products of p-points, there is an isomorphic  ultrafilter of $\vec{\mathcal{U}}$-trees which  supports basic cofinal maps in the topological space where the $\vec{\mathcal{U}}$-trees exist. 
We begin with the relevant definitions and facts.

\begin{notn}\label{notation.cross}
Let $\mathcal{U}$, $\mathcal{V}$, and $\mathcal{U}_n$ ($n<\om$) be ultrafilters.
We define the notation for the following ultrafilters. 
\begin{enumerate}
\item
$\mathcal{U}\cdot\mathcal{V}= \{A\sse\om\times\om:\{i\in\om:\{j\in\om:(i,j)\in A\}\in\mathcal{V}\}\in\mathcal{U}\}$.
\item
$\lim_{n\ra\mathcal{U}}\mathcal{U}_n=\{A\sse\om\times\om:\{n\in\om:\{j\in\om:(n,j)\in A\}\in\mathcal{U}_n\}\in\mathcal{U}\}$.
\end{enumerate}
$\lim_{n\ra\mathcal{U}}\mathcal{U}_n$ is called the \em Fubini product \rm of $\mathcal{U}_n$ over $\mathcal{U}$.
\end{notn}

The Fubini product construction of ultrafilters can be iterated countably many times, each time producing an ultrafilter.
Such an ultrafilter
 can be treated  as an ultrafilter with a front as its base set, which we now make precise.

\begin{defn}[\cite{Argyros/TodorcevicBK}]\label{def.front}
A family $B$ of finite subsets of $\bN$ is called a \em front \rm if 
\begin{enumerate}
\item
$a\not\sqsubseteq b$ whenever $a\not= b$ are in $B$;
\item
$\bigcup B$ is infinite and for every infinite $X\sse\bigcup B$ there exists $b\in B$ such that $b\sqsubseteq X$.
\end{enumerate}
\end{defn}

In standard notation for fronts, we use $\bN$ to denote $\om$.
We use $\bN^{[k]}$ to denote the collection of $k$-element subsets of $\om$.
Naturally then, $\bN^{[<k]}$ denotes the collection of subsets of $\bN$ of size less than $k$,
$\bN^{[\le k]}$ denotes the collection of subsets of $\bN$ of size less than or equal to $k$, and
$\bN^{[<\infty]}$ denotes the collection of finite subsets of $\bN$.
It is easy to check that for each $k<\om$, $\bN^{[k]}$ is a front.

Every front  is lexicographically well-ordered, and hence has a unique rank associated with it, namely the ordinal length of its lexicographical well-ordering.
For example, $\rank(\{\emptyset\})=1$,  $\rank(\bN^{[1]})=\om$, and $\rank(\bN^{[2]})=\om\cdot\om$. 
Suppose $B$ is a front and $M=\bigcup B$, so $B$ is a front on $M$.
For each $n\in M$, define $B_n=\{b\in B:n=\min(b)\}$,
and define $B_{\{n\}}=\{b\setminus \{n\}:b\in B_n\}$.
Then $B=\bigcup_{n\in M} B_n$, and each $B_n=\{\{n\}\cup a:a\in B_{\{n\}}\}$.
For each $n\in M$,
$B_{\{n\}}$ is a front on $M\setminus (n+1)$ with rank strictly less than the rank of $B$.
(See \cite{Argyros/TodorcevicBK}.)

Given any front $B$,
let $C=C(B)$ denote the collection of all proper initial segments of elements of $B$;
that is, $C=\{c\in\bN^{[<\infty]}:\exists \in B\, (c\sqsubset b)\}$.
Let $D=B\cup C$.
Notice that $D$ forms a tree under the partial ordering of initial segments.

\begin{defn}[\cite{TodorcevicBK10}]\label{def.Utree}
Given a front $B$ and  a sequence $\vec{\mathcal{U}}=(\mathcal{U}_c:c\in C)$ of nonprincipal ultrafilters on $\om$,
a {\em $\vec{\mathcal{U}}$-tree} is a tree $T\sse D$ with the property that $\{n:\in\om:c\cup \{n\}\in T\}\in\mathcal{U}_c$ for all $c\in C$.
\end{defn}

\begin{notn}\label{notn.Utrees}
Given a front $B$ and a sequence $\vec{\mathcal{U}}=(\mathcal{U}_c:c\in C)$ of nonprincipal ultrafilters on $\om$,
let $\mathfrak{T}=\mathfrak{T}(\vec{\mathcal{U}})$ denote the collection of all $\vec{\mathcal{U}}$-trees.
For any $c\in C$ and $T\in\mathfrak{T}$, define 
$T| c=\{t\in T:t\sqsubseteq c$ or $c\sqsubseteq t\}$,
 the collection of all nodes in $T$ comparable with $c$.
Let $\mathfrak{T}|c$ denote the collection of all $T|c$, $T\in\mathfrak{T}$.
For any tree $T$, let $[T]$ denote the collection of maximal branches through $T$;
thus $[T]\sse B$.
\end{notn}

The space of all $\vec{\mathcal{U}}$-trees forms a topological ultra-Ramsey space, whether or not the collection $\{[T]: T\in\mathfrak{T}(\vec{\mathcal{U}})\}$ is an ultrafilter.
For more on this topic, see \cite{TodorcevicBK10}.

The following observation was pointed out to us by Todorcevic.

\begin{fact}\label{fact.fubprod}
To every ultrafilter $\mathcal{W}$ which is a countable iteration of Fubini products of ultrafilters, there corresponds a front $B$ 
and ultrafilters $\vec{\mathcal{U}}=(\mathcal{U}_c :c\in C(B))$, such that 
$\{[T]:T\in\mathfrak{T}(\vec{\mathcal{U}})\}$
is  an ultrafilter on the base set $B$ isomorphic to $\mathcal{W}$. 
\end{fact}

\begin{proof}
Suppose that $\mathcal{W}=\lim_{n\ra\mathcal{V}}\mathcal{W}_n$.
Define $\mathcal{U}_{\emptyset}=\mathcal{V}$, and $\mathcal{U}_{\{n\}}=\mathcal{U}_n$ for each $n<\om$.
Let $B=\bN^{[2]}$ and $\vec{\mathcal{U}}=(\mathcal{U}_c:c\in \bN^{\le 1})$.
Let $\Delta$ denote the upper triangle $\{(m,n):m<n<\om\}$ on $\om\times\om$.
Let $\theta:\Delta\ra \bN^{[2]}$ by $\theta((m,n))=\{m,n\}$.
Then $\theta$ witnesses that
$\mathcal{W}\re\Delta=\{W\in \mathcal{W}:W\sse\Delta\}$ is isomorphic to $\{[T]:T\in\mathfrak{T}(\vec{\mathcal{U}})\}$.
Since $\mathcal{W}\re\Delta$ is isomorphic to the original $\mathcal{W}$, we have that $\{[T]:T\in\mathfrak{T}(\vec{\mathcal{U}})\}$ is isomorphic to $\mathcal{W}$.

For the inductive step, let $\mathcal{W}=\lim_{n\ra \mathcal{V}}\mathcal{W}_n$ be a Fubini product such that the  Fact holds for each $\mathcal{W}_n$, $n<\om$.
Thus,  for each $n$, there are a front $B(n)$ and ultrafilters $\mathcal{U}_c(n)$, $c\in C(B(n))$ such that $\mathcal{W}_n$ is isomorphic to the collection 
$\{[T]:T\in \mathfrak{T}(\mathcal{U}_c(n):c\in C(B(n)))\}$.
In the standard way, we modify the fronts  and then glue them together to obtain a new front which provides a base.
Let $B_{\{n\}}$ be the front on $\bN\setminus (n+1)$ which is the isomorphic image of $B(n)$,
via the isomorphism $\varphi_n:\om\ra \bN\setminus(n+1)$ by $\varphi_n(m)=n+1+m$.
Given $n$ and $c\in C(B_{\{n\}})$, 
let $\mathcal{U}_{\{n\}\cup c}$ denote $\mathcal{U}_{\varphi^{-1}_n(c)}(n)$.
Let $\mathcal{U}_{\emptyset}=\mathcal{V}$.
Finally let $B=\bigcup_{n<\om}\{\{n\}\cup b:b\in  B_{\{n\}}\}$.
Then $B$ is a front on $\bN$.
Further,
 $\{[T]:T\in\mathfrak{T}(\mathcal{W}_c:c\in C(B))\}$ is isomorphic to 
$\mathcal{U}$.
\end{proof}

For more on the iterative process of building new fronts from old ones, we refer the reader to \cite{Argyros/TodorcevicBK}.

\begin{defn}\label{defn.lo.on.front}
Let $\prec$ denote the following linear ordering on $\bN^{[<\infty]}$.
Given any $a,b\in\bN^{[<\infty]}$ with $a\ne b$,
enumerate them in increasing order as
$a=\{a_1,\dots,a_m\}$ and $b=\{b_1,\dots,b_n\}$.
Here  $m$ equals the cardinality of $a$ and $n$ equals the cardinality of $b$,
and no comparison between $m$ and $n$ is assumed.
Define 
$a\prec b$ iff 
\begin{enumerate}
\item
$a=\emptyset$; or
\item
$\max(a)<\max(b)$; or 
\item
$\max(a)=\max(b)$ and $a_i<b_i$, where $i$ is the least such that $a_i\ne b_i$.
\end{enumerate}
\end{defn}
Thus, $(\bN^{[<\infty]},\prec)$ is ordered as follows:
$\emptyset\prec\{0\}\prec\{0,1\}\prec\{1\}\prec\{0,1,2\}\prec\{0,2\}\prec\{1,2\}\prec\{2\}\prec\{0,1,2,3\}\prec\dots$.
Moreover, for each $k<\om$, the set $\{c\in\bN^{[<\infty]}:\max(c)=k\}$ forms a finite interval in $(\bN^{[<\infty]},\prec)$.

We would like to have a theorem stating that  any iterated Fubini product of p-points has basic Tukey reductions.
The following example illustrates why this is impossible.
Let $\mathcal{U}$ and $\mathcal{V}$ be any ultrafilters, p-points or otherwise, and let $f:\om\times\om\ra\om$ be given by $f((n,j))=n$.
Then
$f:\mathcal{U}\cdot\mathcal{V}\ra\mathcal{U}$    is a monotone cofinal map.
However, there is no cofinal $\mathcal{X}\sse\mathcal{U}\cdot\mathcal{V}$ for which $f\re\mathcal{X}$ is basic:
Given any linear ordering of $\om\times\om$ isomorphic to $\om$,
$f\re\mathcal{X}$ cannot be generated by an initial segment preserving map.  
For any $X\in\mathcal{U}\cdot\mathcal{V}$, for each $n$, there is no bound on the finite initial segment of $X$ needed in order to   know whether or not there is a $j$ for which $(n,j)$ is in  $X$.

However, 
in the space of $\vec{\mathcal{U}}$-trees,
we do obtain a theorem analogous to Theorem \ref{thm.ppoint}.
Given a front $B$ and a collection $\vec{\mathcal{U}}=(\mathcal{U}_c:c\in C(B))$ of p-points,
recall that the collection of $\vec{\mathcal{U}}$-trees has basis $D=B\cup C$.
Ordering $D$ by $\prec$, we obtain a linear ordering isomorphic to $\om$.
Thus, $2^D$ is isomorphic to the Cantor space.
It is in this topological space that we attain basic Tukey reductions.

\begin{thm}\label{thm.allFubProd_p-point_cts}
Let $B$ be any front and $\vec{\mathcal{U}}=(\mathcal{U}_c:c\in C)$ be a sequence of p-points.
Let $\mathfrak{T}$ denote the collection of $\vec{\mathcal{U}}$-trees.
If $f:\mathfrak{T}\ra\mathcal{V}$ is a monotone cofinal function, 
then there is a $\vec{\mathcal{U}}$-tree $\tilde{T}$ such that $f\re(\mathfrak{T}\re\tilde{T})$ is basic, in the topological space $2^D$.
\end{thm}

\begin{proof}
For each $k<\om$ and any $A\sse D$, let $A\re k$ denote $\{a\in A:\max{a}<k\}$.
The proof is structured in three stages.
In Stage 1 we modify an argument from the proof of Theorem \ref{thm.ppoint}
 to fit the setting of trees.
For each $k<\om$ and $A\sse D\re k$, we obtain $\vec{\mathcal{U}}$-trees $S_k(A)$ such that for any $\vec{\mathcal{U}}$-tree $T\sse S_k(A)$ with $T\re k=A$, $j\in f(T)$ iff $j\in f(S_k(A))$, for all $j\le k$.
In Stage 2, we construct functions  $g_c$, $c\in C$, any two of which eventually line up in the following way.
For each $l$,  there are $m_0,\dots,m_l$ such that $g_{c_0}(2m_0)=g_{c_1}(2m_1)=\dots=g_{c_l}(2m_l)$.
We shall say that these functions $g_c$ \em mesh\rm.
These functions will be used to obtain sets $U_{c_j}\in\mathcal{U}_{c_j}$  for all $j\le l$, so that each   $U_{c_j}$ has an empty interval starting at $g_{c_0}(2m_0)$ and going up until $g_{c_j}(2m_j+1)$.
The tree $T^*$ will be defined as the tree such that for each node $c\in T^*$, the set of immediate extensions of $c$ in $T^*$ is exactly $U_c$.
In Stage 3,  we 
  thin $T^*$  to a better tree $\tilde{T}$ in a sort of reverse thinning process. 
This involves simultaneously widening the spaces before certain common points  in the ranges of finite collections of $g_c$.
We use the $h_c$ functions in Lemma \ref{lem.hconditions} to finish the thinning process to obtain $\tilde{T}$ such that $f$ is basic on $\mathfrak{T}\re\tilde{T}$, in the topological space $2^D$.
Having presented the outline, we now begin the proof.

Let   $\{c_i:i<\om\}$ enumerate $C$ in $\prec$-increasing order.
For any $k<\om$ and any $A\sse D$, let
 $A\re k$ denote the collection of $d\in A$ with $\max(d)<k$.
 For $a\sqsubseteq c\in C$ such that $c\in T(a)\in\mathfrak{T}|a$, 
 let $U_c(T(a))$ denote $\{l:l>\max(c)$ and $c\cup \{l\}\in T(a)\}$, the collection of $l$ which extend $c$ into $T(a)$.
There is no abuse of notation here, because $U_c(T(a))$ truly is a member of $\mathcal{U}_c$.
\vskip.1in

\noindent \bf Stage 1\rm. 
Choose trees $T_k(c)\in\mathfrak{T}|c$ which combine to build the desired $\vec{\mathcal{U}}$-trees $S_k(A)$ mentioned above.

\noindent \underline{Step 0}. \rm
$C\re 0 =\{\emptyset\}$.
Choose $T_0(\emptyset)$ such that (1), (2), and $(*_0)$ hold.
\begin{enumerate}
\item[(1)]
$T_0(\emptyset)\in\mathfrak{T}$.
\item[(2)]
$\min(W_{\emptyset}(T(\emptyset)))>0$.
\end{enumerate}
Let $S_0(\{\emptyset\})$ denote $T_0(\emptyset)$.
\begin{enumerate}
\item[$(*_0)$]
For each $T\in\mathfrak{T}$ such that $T\sse S_0(\{\emptyset\})$,
$0\in f(T)$ iff
$0\in f(S_0(\{\emptyset\}))$.
\end{enumerate}

Now, suppose $k\ge 1$ and for all $l<k$, we have done the construction for Step $l$.

\noindent\underline{Step k}. \rm
For each $c\in C\re k$, 
choose $T_k(c)$ such that (1) - (3),  $(S_k)$, and  $(*_k)$ hold.
\begin{enumerate}
\item[(1)]
$T_k(c)\in\mathfrak{T}| c$.
\item[(2)]
$\min(W_{c}(T(c)))>k$.
\item[(3)]
Suppose $c\in C\re k$,  $l<k$,
$a\in C\re l$, $a\sqsubseteq c$, and $c\in T_l(a)$.
Then $T_k(c)\sse T_l(a)$.
\end{enumerate}
(2) says that for each $d\in T_{k}(c)$ such that $d\sqsupset c$, $\min(d\setminus c)>k$.
(3) says that if $c$ appears in a tree $T_l(a)$ already chosen before step $k$, and $c$ equals or extends the stem $a$  of $T_l(a)$, then  $T_k (c)$ is a subtree of $T_l(a)$.

For each $A\sse D\re k$ closed under initial segments,
we define 
\begin{equation}
S_k(A)
=
\bigcup\{T_k(c):c\in A\cap C\}\cup(A\cap B).
\end{equation}

\begin{enumerate}
\item[$(S_k)$]
For each $A\sse D\re (k-1)$ closed under initial segments, $S_k(A)\sse S_{k-1}(A)$.
\end{enumerate}

\begin{enumerate}
\item[$(*_{k})$]
Suppose $A\sse D\re k$  is closed under initial segments, $T\in\mathfrak{T}$, and $T\sse S_k(A)$.
Then
$k\in f(T)$
iff
$k\in f(S_{k}(A))$.
\end{enumerate}

$(S_k)$ for all $k<\om$ implies  the following $(S)$.
$(S)$ along with $(*_k)$ for all $k<\om$ implies  $(*)$,  which will be essential in the proof that the function $f$ will be basic when restricted below a certain tree.

\begin{enumerate}
\item[$(S)$]
For all $l<k<\om$, for all $A\sse D\re l$ closed under initial segments, 
then $S_l(A)\sse S_k(A)$.
\end{enumerate}

\begin{enumerate}
\item[$(*)$]
Suppose $l<k<\om$,
$T\in\mathfrak{T}$ with $T\sse S_k(A)$, and 
$j\le l$.
Then $j\in f(T)$ iff $j\in f(S_k(A))$ iff $j\in f(S_l(A))$.
\end{enumerate}

We now show how to choose the trees for Stage 1.
Some useful notation is the following: For each $k<\om$ and  each $A\sse D\re k$ closed under initial segments, 
let 
\begin{equation}
D_k(A)=\{d\in D:d\in A\mathrm{\  or \ }\exists a\in A\, (d\sqsupset a\mathrm{\ and\ }\min(d\setminus a)>k)\}.
\end{equation}

For Step $0$,
 if there is a $\mathfrak{T}$-tree $R$ for which $0\not\in f(R)$, then let $T_0(\emptyset)$ be such a tree with $T_0(\emptyset)\sse D_0(\emptyset)$.
 Otherwise, for every $T\in \mathfrak{T}$, $0\in T$.
In this case,  let $T_0(\emptyset)=D_0(\{\emptyset\})$.
Then (1) and (2) are satisfied.
Further,  $f$ being monotone implies that $S_0(\emptyset)=T_0(\emptyset)$ satisfies $(*_0)$.
For suppose $T\in\mathfrak{T}$ and $T\sse T_0(\emptyset)$.
 If $0\in f(T_0(\emptyset))$,
 then  there is no $R\in\mathfrak{T}$ for which $0\not\in f(R)$, so
$0$ must also be in $f(T)$.
Conversely, if $0\not\in f(T_0(\emptyset))$, then $f$ being monotone and $T\sse T_0(\emptyset)$ imply that also $0\not\in f(T)$.

For the general Step $k\ge 1$, 
list all nonempty subsets of $D\re k$ which are closed under initial segments as $\lgl A_l:l< p_k\rgl$, where $p_k$ is the number of such sets.
For each $l<p_k$,
if there is an $R\in\mathfrak{T}$ such that $R\cap (D\re k)=A_l$ and $k\not\in f(R)$,
take such an $R$ such that also
 $R\sse D_k(A_l)$,
and label it $R_l$.
Otherwise, let $R_l=D_k(A_l)$.
For each $c\in C\re k$, define
\begin{equation}
T_k(c)=\bigcap\{R_l|c:l<p_k,\ c\in R_l \}\cap\bigcap\{T_j(a):j<k,\ a\sqsubseteq c,\ c\in T_j(a)\}.
\end{equation}

For each $c\in C\re k$,
$T_k(c)$ is an intersection of finitely many trees in $\mathfrak{T}|c$, and hence is itself in $\mathfrak{T}|c$; thus, (1) holds.
For $l<p_k$ such that $A_l=\{a\in C\re k:a\sqsubseteq c\}$,
we have $T_k(c)\sse R_l\sse D_k(A_l)$; hence (2) holds.
The rightmost intersection in the definition of $T_k(c)$ ensures that (3) and $(S_k)$ hold.
Moreover, for  each $l<p_k$ such that
$c\in A_l$,
$T_k(c)\sse R_l|c\sse R_l$.
Therefore,
\begin{equation}
S_k(A_l)=\bigcup\{T_k(c):c\in A_l\}\cup(A_l\cap B)
\sse R_l.
\end{equation}
Hence, $k\in f(S_k(A_l))$ iff $k\in f(R_l)$,
since $S_k(A_l)\sse R_l$ and $S_k(A_l)\re k=A_l=R_l\re k$.
By the definition of $R_l$, it follows that for each $\vec{\mathcal{U}}$-tree $T\sse S_k(A_l)$ with $T\re k=A_l$, $k\in f(T)$ iff $k\in S_k(A_l)$.
Thus, $(*_k)$ holds.
\vskip.1in

\noindent \bf Stage 2. \rm
For each $k$ and $c\in C\re k$, let $U_{c}(T_k(c))$ denote $\{l:c\cup\{l\}\in T_k(c)\}$.
By (1) and (2),  we have $U_c(T_k(c))\in\mathcal{U}_c$ and $\min(U_c(T_k(c)))>k$.
By (3), $U_{\emptyset}(T_0(\emptyset))\contains U_{\emptyset}(T_1(\emptyset))\contains\dots$.
Since $\mathcal{U}_{\emptyset}$ is a p-point,
 there is a $U^*_{\emptyset}\in \mathcal{U}_{\emptyset}$ such that 
$U^*_{\emptyset}\sse U_{\emptyset}(T_0(\emptyset))$ and 
 for each $k<\om$, $U^*_{\emptyset}\sse^* U_{\emptyset}(T_k(\emptyset))$.
Let  $n_0=2$.
Take $n_1>n_0$ such that $U^*_{\emptyset}\setminus n_1\sse U_{\emptyset}(T_{n_0}(\emptyset))$.
In general, take $n_{i+1}>n_i$ such that $U^*_{\emptyset}\setminus n_{i+1}\sse U_{\emptyset}(T_{n_i}(\emptyset))$.
Define $g_{\emptyset}:\om\ra\om$ by $g_{\emptyset}(i)=n_i$.
Either $\bigcup_{i<\om}[n_{2i},n_{2i+1})$ or $\bigcup_{i<\om}[n_{2i+1},n_{2i+2})$ is in $\mathcal{U}_{\emptyset}$.
Without loss of generality, assume that $\bigcup_{i<\om}[n_{2i+1},n_{2i+2})\in\mathcal{U}_{\emptyset}$.
(If it is not, then letting $n'_i=n_{i+1}$,  $\bigcup_{i<\om}[n'_{2i+1},n'_{2i+2})$ will be in $\mathcal{U}_{\emptyset}$.
The point is that we want to have uniform  indexing for this and successive stages by requiring the upper bound of the interval to have  even index so as to avoid more subindexing than is necessary.)
Let $U_{\emptyset}=U^*_{\emptyset}\cap (\bigcup_{i<\om}[n_{2i+1},n_{2i+2}))$.

For $n\ge 0$,
suppose we 
have chosen $g_{c_i}$ and $U_{c_i}$ for all $i\le n$.
For each $k\ge n+1$, 
 $T_k(c_{n+1})$ is defined,
and
 $U_{c_{n+1}}(T_k(c_{n+1}))\in\mathcal{U}_{c_{n+1}}$.
By our construction, 
$U_{c_{n+1}}(T_{n+1}(c_{n+1}))\contains U_{c_{n+1}}(T_{n+2}(c_{n+1})) \contains\dots$.
Since $\mathcal{U}_{c_{n+1}}$ is a p-point, there is a $U^*_{c_{n+1}}\in \mathcal{U}_{c_{n+1}}$ such that 
$U^*_{c_{n+1}}\sse U_{c_{n+1}}(T_{n+1}(c_{n+1}))$, and 
for all $k\ge n+1$, 
$U^*_{c_{n+1}}\sse^* U_{c_{n+1}}(T_k(c_{n+1}))$.
Let $g_{c_{n+1}}:\om\ra\om$ be a strictly increasing function such that 
$g_{c_{n+1}}(0)\ge n+1$ and 

\begin{enumerate}
\item[$(\dagger_n)$]
$   \forall m\  \exists i\ (g_{c_{n+1}}(m)=g_{c_n}(2i))$;
\end{enumerate}
and
\begin{enumerate}
\item[$(\ddagger)$]
$  \forall m,
U^*_{c_{n+1}}\setminus g_{c_{n+1}}(m+1)\sse
U_{c_{n+1}}(T_{g_{c_{n+1}}(m)}(c_{n+1}))$.
\end{enumerate}
Without loss of generality, 
\begin{equation}
\bigcup_{i<\om}[g_{c_{n+1}}(2i+1),g_{c_{n+1}}(2i+2))\in\mathcal{U}_{c_{n+1}}.
\end{equation}
Let 
\begin{equation}
U_{c_{n+1}}=U^*_{c_{n+1}}\cap (\bigcup_{i<\om}[g_{c_{n+1}}(2i+1),g_{c_{n+1}}(2i+2))).
\end{equation}
Note that $(\dagger_n)$ for all $n<\om$ implies
\begin{enumerate}
\item[$(\dagger)$]
For all $n<n'$, for all $m$, there is an $i$ such that 
$g_{c_n}(2i)=g_{c_{n'}}(m)$.
\end{enumerate}

Define $T^*\sse D$ to be the tree defined by the $U_{c}$, $c\in C$, as follows.
$\emptyset\in T^*$.
For each $n<\om$, $\{n\}\in T^*$ iff $n\in U_{\emptyset}$.
In general, if $c\in T^*$, then $\{l:c\cup\{l\}\in T^*\}=U_c$.
Thus, $T^*\in\mathfrak{T}$.
\vskip.1in

\noindent\bf Stage 3. \rm
Let $C^*=C\cap T^*$,
 and let $C_*=\{c\in C^*:\exists c'\in C^*\ (c'\sqsupset c)\}$.
So $C_*$ consists of all elements of $C^*$ which are not maximal in $C^*$. 
Define a strictly increasing sequence $(j_i)_{i<\om}$ as follows.
Let $j_0=g_{\emptyset}(0)$.
Take $j_1>j_0$ such that for each $c\in C^*\re j_0$, there is an $m$ such that $j_0<g_c(2m-1)$ and $g_c(2m)=j_1$.
In general, take $j_{i+1}>j_i$ such that for each $c\in C^*\re j_i$, there is an $m$ such that $j_i<g_c(2m-1)$ and $g_c(2m)=j_{i+1}$.

\begin{lem}\label{lem.hconditions}
Suppose $g_c$, $U_c\in\mathcal{U}_c$, $c\in C^*$, and $(j_i)_{i<\om}$
satisfy the following.
\begin{enumerate}
\item[(a)]
For each $k<\om$ and $c\in C^*\re k$, 
$T_k(c)$ and the corresponding sets $U_c(T_k(c))\in\mathcal{U}_c$ satisfy (1) - (3), $(S_k)$, and $(*)_k$, as in Stage 1.
\item[(b)]
$U^*_c\in\mathcal{U}_c$ and the $g_c$ are chosen as in Stage 2, and satisfy $(\dagger)$ and $(\ddagger)$.
\item[(c)]
For each $i<\om$ and each $c\in C^*$ with $\max(c)<j_i$,
there is an $m$ such that $g_c(2m)=j_{i+1}$
and $g_c(2m-1)>j_i$.
\item[(d)]
$U_c=U_c^*\cap Z_c$, where $Z_c=\bigcup_{i<\om}[g_c(2i+1),g_c(2i+2))\in\mathcal{U}_c$.
\end{enumerate}

Then there are functions $h_c$, $c\in C_*$, which satisfy the following.
For each $c\in C_*$,
\begin{enumerate}
\item[(i)]
For each $m$, $h_c(m)=j_i$, for some $i$.
\item[(ii)]
If $c$ is not maximal in $C_*$, then for each $l<l'$ in $U_c$
and each $m'$,
there is an $m$ such that 
$h_{c\cup\{l'\}}(m')=h_{c\cup\{l\}}(2m)$.
\item[(iii)]
If $c$ is not maximal in $C_*$,
then for  each $m$ and each $l\in U_c\cap h_c(m)$,
there is an $m_l$ such that 
$h_{c\cup\{l\}}(2m_l)=h_c(m+1)$,
and $h_{c\cup\{l\}}(2m_l-1)>h_c(m)$.
\item[(iv)]
For each $m\ge 1$ and each $a\sqsubset c$ with $\max(c)<h_{a}(2m-1)$,
there is an $m'$ such that 
$h_c(2m')=h_{a}(2m)$.
\item[(v)]
Letting $Y_c=\bigcup_{i<\om}[h_c(2i),h_c(2i+1))$, we have
$Y_c\in\mathcal{U}_c$.
\end{enumerate}
\end{lem}

Before proving  Lemma \ref{lem.hconditions}, we first use it  to prove the Theorem.
Given the $Y_c$, $c\in C_*$, from Lemma  \ref{lem.hconditions} (v),
define 
 $\tilde{U}_c=U_c\cap Y_c$.
 For $c\in C^*\setminus C_*$, let $\tilde{U}_c=U_c$.
Define the $\vec{\mathcal{U}}$-tree $\tilde{T}$ as follows.
$\emptyset\in\tilde{T}$, and $\{l:\{l\}\in\tilde{T}\}=\tilde{U}_{\emptyset}$.
Given $c$ already defined to be in $\tilde{T}$,
let $\{l:c\cup\{l\}\in \tilde{T}\}=\tilde{U}_c$.
Thus, if $c\in \tilde{T}$, then
$U_c(\tilde{T})=\tilde{U}_c$.

By (c), (i), and (iv),
for all $a\sqsubset c$ in $\tilde{T}$
and 
$m\ge 1$  with $\max(c)<h_{a}(2m-1)$,
there is an $m'$ such that $h_c(2m')=h_a(2m)$.
Moreover,
 there is an $i$ such that $h_c(2m')=g_c(2i)$.
Therefore,
 \begin{equation}
 \tilde{U}_c\cap[h_{a}(2m),g_c(2i+1))
 =\tilde{U}_c\cap [g_c(2i),g_c(2i+1))=\emptyset.
 \end{equation}
Hence, 
\begin{equation}
\tilde{U}_c\setminus h_{a}(2m)\sse
U^*_c\setminus g_c(2i+1)
\sse 
U_c(T_{g_c(2i)})
=U_c(T_{h_{a}(2m)}).
\end{equation}
In particular,
for each $c$ with $\max(c)<h_{\emptyset}(2m)$,
there is an $i$ such that 
$g_c(2i)=h_{\emptyset}(2m)$.
So $\tilde{U}_c\cap[h_{\emptyset}(2m),g_c(2i+1))=\emptyset$.
Hence, 
\begin{equation}
\tilde{U}_c\setminus h_{\emptyset}(2m)
\sse
U^*_c\setminus g_c(2i+1)
\sse U_c(T_{g_c(2i)})
=U_c(T_{h_{\emptyset}(2m)}).
\end{equation}

\begin{claimn}\label{claim.Tinitseg}
Take $(n_k)_{k<\om}$ to be a strictly increasing sequence such that $h_{\emptyset}(2n_k)>k$, and 
let $\tilde{n}_k=h_{\emptyset}(2n_k+2)$.
Then
for any $\vec{\mathcal{U}}$-tree $T\sse\tilde{T}$,
 $k\in f(T)$ iff
$k\in f(S_{\tilde{n}_k}(A))$,
where $A=T\re \tilde{n}_k$.
\end{claimn}

\begin{proof}
Let $T\sse\tilde{T}$ be a $\vec{\mathcal{U}}$-tree, and let $k<\om$ be given.
Let $n$ denote $n_k$ and   $\tilde{n}$ denote $\tilde{n}_k$. 
Let $A=T\re\tilde{n}$.
Since,
$\tilde{U}_{\emptyset}\cap[h_{\emptyset}(2n+1),h_{\emptyset}(2n+2))=\emptyset$, it follows that every $l_0\in\tilde{U}_{\emptyset}\cap \tilde{n}$ is actually less than $h_{\emptyset}(2n +1)$.
Hence, $h_{\{l_0\}}(2m_{\{l_0\}}+2)=\tilde{n}$ for some $m_{\{l_0\}}$, by (iii)
By (c) and (i), there is a $p_{\{l_0\}}$ such that $g_{\{l_0\}}(2p_{\{l_0\}})=\tilde{n}$.
Since $\tilde{U}_{\{l_0\}}\cap[h_{\{l_0\}}(2m_{\{l_0\}}) +1),h_{\{l_0\}}(2m_{\{l_0\}} +2))=\emptyset$,
each $l_1$ in $\tilde{U}_{\{l_0\}}\cap \tilde{n}$ is actually less than $h_{\{l_0\}}(2m_{\{l_0\}}+1)$.
So by (iii) there is some $m_{\{l_0,l_1\}}$ such that $h_{\{l_0,l_1\}}(2m_{\{l_0,l_1\}}+2)=\tilde{n}$.
By (c) and (i), there is a $p_{\{l_0,l_1\}}$ such that $g_{\{l_0\}}(2p_{\{l_0,l_1\}})=\tilde{n}$.
This process continues  as long as $\max(c)<\tilde{n}$.
That is, for $c=\{l_0,\dots,l_r\}$ with $l_r<\tilde{n}$,
there is an $m_c$ such that $h_c(2m_c+2)=\tilde{n}$,
and there is a $p_c$ such that $g_c(2p_c)=\tilde{n}$. 
Thus, we have
\begin{equation}
h_c(2m_c +1)<h_c(2m_c +2)=\tilde{n}= g_c(2p_c)<g_c(2p_c+1).
\end{equation}
It follows that
\begin{equation}
\tilde{U}_c\cap[h_c(2m_c +1),g_c(2p_c+1))=\emptyset.
\end{equation}
The point  is that  for each $c\in T\cap (C\re\tilde{n})$, 
\begin{equation}
U_c(\tilde{T})\setminus \tilde{n}
\sse 
U_c(T_{\tilde{n}}),
\end{equation}
and
 for each $c\in T\cap C$ with $\max(c)\ge\tilde{n}$,
\begin{equation}
T\re c\sse T_{\tilde{n}}|c.
\end{equation} 
Since 
\begin{equation}
T=A\cup\bigcup\{T|c:c\in T\cap C\, (\max(c)\ge\tilde{n})\},
\end{equation}
it follows that
\begin{enumerate}
\item[$(\star)$] 
 $k\in f(T)$ iff
$k\in f(S_{\tilde{n}}(A))$.
\end{enumerate}
\end{proof}

Ordering $D$ by $\prec$ as $\lgl d_i:i<\om\rgl$, we obtain a linear ordering isomorphic to $\om$.
We consider the space $2^D$ of all characteristic functions of subsets of $D$  as a topological space isomorphic to the Cantor space.

\begin{claimn}\label{claim.fcts}
 $f\re(\mathfrak{T}\re\tilde{T})$ is a basic map, generated by a monotone, level and initial segment preserving map $\hat{f}:\bigcup_{k<\om}2^{D\re\tilde{n}_k}\ra 2^{<\om}$.
\end{claimn}

\begin{proof}
Given $k<\om$, a finite subset $A\sse \tilde{T}\re \tilde{n}_k$,  and $s(A)$ the characteristic function of $A$ with domain $D\re \tilde{n}_k$,
define 
$\hat{f}(s(A))$
to be the characteristic function with domain $k+1$ of 
$f(S_{\tilde{n}_k}(A))\cap (k+1)$.
Then $\hat{f}$ is level and initial segment preserving.
Furthermore, for each $T\in\mathfrak{T}\re\tilde{T}$,
\begin{equation}
f(T)=\bigcup_{k<\om}d(\hat{f}(s(T\re \tilde{n}_k))).
\end{equation}
Thus, $f$ is continuous on $\mathfrak{T}\re\tilde{T}$.
\end{proof}

This concludes the proof of Theorem \ref{thm.allFubProd_p-point_cts}, 
modulo the proof of Lemma \ref{lem.hconditions} which we now give.

\begin{proof} (\it of Lemma \ref{lem.hconditions}). \rm
The proof is by induction on the lexicographical rank of the front.
Suppose $B=\bN^{[2]}$  corresponding to a Fubini product $\lim_{n\ra\mathcal{U}_{\emptyset}}\mathcal{U}_{\{n\}}$,
and the hypotheses (a) - (d) are satisfied.
Define $h_{\emptyset}:\om\ra\om$ as follows.
Let $h_{\emptyset}(0)=g_{\emptyset}(0)$.
Take $h_{\emptyset}(1)>h_{\emptyset}(0)$ such that 
$h_{\emptyset}(1)=j_{i_1}$ for some $i_1$, and 
for each $l<h_{\emptyset}(0)$ there is an $m$ such that $g_{\{l\}}(2m)=h_{\emptyset}(1)$ and $g_{\{l\}}(2m-1)>h_{\emptyset}(0)$.
In general, 
take
$h_{\emptyset}(k+1)>h_{\emptyset}(k)$ such that
\begin{enumerate}
\item
$h_{\emptyset}(k+1)=j_{i_{k+1}}$ for some $i_{k+1}$;
and 
\item
for each $l<h_{\emptyset}(k)$ there is an $m$ such that $g_{\{l\}}(2m)=h_{\emptyset}(k+1)$ and $g_{\{l\}}(2m-1)>h_{\emptyset}(k)$.
\end{enumerate}

Without loss of generality, suppose that $Y_{\emptyset}=\bigcup_{i<\om}[h_{\emptyset}(2i),h_{\emptyset}(2i+1))\in\mathcal{U}_{\emptyset}$.
Let $\tilde{U}_{\emptyset}=U_{\emptyset}\cap Y_{\emptyset}$.
Then (i) and (v) hold.  (ii) - (iv) are trivially satisfied, since $C_*=\emptyset$ when $B=\bN^{[2]}$.

Now suppose that $B$ is a front of lexicographical rank $\al$  and that the Lemma holds for all 
fronts of smaller rank.
For each $n<\om$,
let $B_{n}=\{b\in B:\min(b)=n\}$.
Note that $B_{n}$ is isomorphic to $B_{\{n\}}:=\{b\setminus \{n\}:b\in B_n\}$, which is a front on $\bN\setminus (n+1)$.
Thus, the induction hypothesis applies to $B_n$.
Enumerate  the elements of $U_{\emptyset}$ as $l_0<l_1<l_2<\dots$.
Define $C_{l_n}$ to be $\{c\in C : c\sqsupset \{l_n\}\}$.

Use the induction hypothesis on $B_{l_0}$ with the sequence $(j_i)_{i<\om}$
 to find meshing functions $h_c$ which satisfy (i) - (v) with regard to $(j_i)_{i<\om}$ for each $c\in C_{l_0}$. 
Next,  define $j^1_i=h_{\{l_0\}}(2i)$, for each $i<\om$. 
Use the induction hypothesis on $B_{l_1}$ with the sequence $(j^1_i)_{i<\om}$ to obtain meshing functions $h_c$
for each $c\in C_{l_1}$ which satisfy (i) - (v) with regard to $(j^1_i)_{i<\om}$.
In general, define  $j^{n+1}_i=h_{\{l_n\}}(2i)$, $i<\om$, 
and use the sequence $(j^{n+1}_i)_{i<\om}$
with the  induction hypothesis on $B_{l_{n+1}}$ to find meshing functions $h_c$ for each $c\in C_{l_{n+1}}$ which satisfy (i) - (v) with regard to $(j^{n+1}_i)_{i<\om}$.
This part of the construction yields (ii) for $B$.

Finally, construct $h_{\emptyset}:\om\ra\om$ to mesh with  all  the $h_{\{l_n\}}$, $n<\om$, as follows.
Let $h_{\emptyset}=h_{\{l_0\}}(2)$.
Given $h_{\emptyset}(i)$,
take $h_{\emptyset}(i+1)$ to be some $j^{h_{\emptyset}(i)}_{p}
>h_{\emptyset}(i)$ for some $p$ such that for each $c\in C^*\setminus\{\emptyset\}$ with $\max(c)<h_{\emptyset}(i)$,
there is an $m_c$ such that $h_c(2m_c)=h_{\emptyset}(i+1)$ and $h_c(2m_c-1)>h_{\emptyset}(i)$.
Without loss of generality, assume that $Y_{\emptyset}=\bigcup_{i<\om}[h_{\emptyset}(2i),h_{\emptyset}(2i+1))\in\mathcal{U}_{\emptyset}$.
Let $\tilde{U}_{\emptyset}=U_{\emptyset}\cap Y_{\emptyset}$.
Hence, (v) holds.
Since each sequence $(j^{k+1}_i)_{i<\om}$ is a subsequence of $(j^k_i)_{i<\om}$ and by the definition of $h_{\emptyset}$, we have satisfied (i).
(iii) and (iv) follow from the application of the induction hypothesis and the definition of $h_{\emptyset}$.
\end{proof}

Thus, the proof of Theorem \ref{thm.allFubProd_p-point_cts} is now complete.
\end{proof}

\begin{rem}\label{rem.stable}
The same construction can be carried out if  the $\mathcal{U}_c$ are stable ordered union ultrafilters on the base set $\FIN=[\om]^{<\om}\setminus\{\emptyset\}$, with the correct notion of front for this setting.
It follows from Theorem \ref{thm.allFubProd_p-point_cts} that every ultrafilter which is Tukey reducible to some countable iteration of Fubini products of p-points or stable ordered-union ultrafilters has Tukey type of cardinality $\mathfrak{c}$.
\end{rem}

One can check that the analogue of the first half of Theorem \ref{thm.PropACtsMaps} and the analogue of Theorem \ref{thm.PropB} 
also hold in the the setting of $\vec{\mathcal{U}}$-trees.
Continuity is interpreted in terms of the stronger convergence in such a space; thus, we do not obtain continuous maps in the classical sense.
However, we still do obtain monotone, level and initial segment  preserving maps, thus the following generalization of the Main Theorem.




\begin{thmMTUT}\label{thm.MainThmUtrees}
Let $B$ be a front, $\vec{\mathcal{U}}=(\mathcal{U}_c:c\in C)$ be a sequence of p-points, and
$\mathcal{W}$ be an ultrafilter on $\om$
such that $\mathcal{W}\le_T\mathfrak{T}(\vec{\mathcal{U}})$.
Then for each ultrafilter $\mathcal{V}\le_T\mathcal{W}$ and each monotone cofinal map $h:\mathcal{W}\ra\mathcal{V}$,
there is a monotone, level and initial segment preserving map  which generates $h$.
\end{thmMTUT}

\begin{cor}\label{cor.c}
Suppose $\mathcal{W}$ is an ultrafilter on $\om$ which is Tukey reducible to  $\mathfrak{T}(\vec{\mathcal{U}})$, where $\vec{\mathcal{U}}=(\mathcal{U}_c:c\in C)$ and each $\mathcal{U}_c$ is a p-point.
Then $\mathcal{W}$ has Tukey type of cardinality $\mathfrak{c}$.
\end{cor}

Finally, a word about the notion of convergence in the space of $\vec{\mathcal{U}}$-trees versus the notion of convergence in an iterated Fubini product of ultrafilters.
Let  $\mathcal{V}$ be some countable iteration of Fubini products of p-points, let $B$ the denote a front  and $\vec{\mathcal{U}}=(\mathcal{U}_c:c\in C)$ the p-points
so that $\mathcal{V}\cong \{[T]:T\in\mathfrak{T}(\vec{\mathcal{U}})\}$.
In general, convergence for $\vec{\mathcal{U}}$-trees $T$ in the topological space $2^D$ is stronger than convergence for elements of $\{[T]:T\in\mathfrak{T}(\vec{\mathcal{U}})\}$ in the topological space $2^B$.
However, in the case that all $\mathcal{U}_c$ are the same selective ultrafilter $\mathcal{U}$
the two notions of convergence coincide on the following cofinal subsets:
the collection $\{B|U:U\in \mathcal{U}\}$, a cofinal subset of $\mathcal{V}$, and $\{D|U:U\in\mathcal{U}\}$, a cofinal subset of $\mathfrak{T}(\vec{\mathcal{U}})$.
(Here $B|U=\{b\in B:b\sse U\}$ and $D|U=\{d\in D:d\sse U\}$.)


\section{Open problems}
\label{sec.op}

We conclude this paper with the following problems.

\begin{problem1}
Determine the class of all ultrafilters which have basic Tukey reductions.
\end{problem1}

Recall 
 that every ultrafilter Tukey below a p-point or Tukey below a stable ordered-union ultrafilter has basic Tukey reductions.
Are there (consistently) any others?
It is likely that those ultrafilters which have basic Tukey reductions will be those that have some p-point-like property, in the sense that for some suitably defined analogue of $\contains^*$, any decreasing sequence of elements of the ultrafilter will have some sort of  pseudointersection.

\begin{problem2}
Does $\mathcal{U}<_T\mathcal{U}_{\mathrm{top}}$ imply that the Tukey type of $\mathcal{U}$ has size $\mathfrak{c}$?
\end{problem2}

Recall that  Theorem \ref{thm.3} implies that the top Tukey type has cardinality $2^{\mathfrak{c}}$.  
On the other hand, all currently considered ultrafilters with Tukey type strictly below $\mathcal{U}_{\mathrm{top}}$ have Tukey type of cardinality $\mathfrak{c}$.
(This follows from work of
 Raghavan in \cite{Raghavan/Todorcevic11} for  basically generated ultrafilters, and from Remark \ref{rem.stable}
 for all Fubini iterations of stable ordered-union ultrafilters.)

\bibliographystyle{model1-num-names}
\bibliography{references1}

\end{document}